\documentclass[12pt]{article}

\newlength{\stefan}
\usepackage{amsmath, amsthm, amsfonts, makeidx}

\usepackage[usenames,dvipsnames]{color}

\DeclareMathSymbol{\subsetneq}{\mathord}{AMSb}{"26}

\newtheorem{lemma}{Lemma}[section]

\newtheorem{theorem}[lemma]{Theorem}

\newtheorem{proposition}[lemma]{Proposition}
\newtheorem{corollary}[lemma]{Corollary}
\theoremstyle{definition}

\newtheorem{definition}[lemma]{Definition}
\newtheorem{example}[lemma]{Example}

\newtheorem{remark}[lemma]{Remark}

\newtheorem{question}[lemma]{Question}

\newcommand{\lp}{\longrightarrow}

\newcommand{\mb}{\mathbb}

\newcommand{\F}{\mb{F}}
\newcommand{\B}{\textup{B}}

\newcommand{\D}{\mathcal{D}}

\newcommand{\Z}{\mb{Z}}
\newcommand{\N}{\mb{N}}
\newcommand{\Q}{\mb{Q}}

\newcommand{\desda}{\Longleftrightarrow}
\newcommand{\Aff}{\mathit{Aff}}

\renewcommand{\ker}{\operatorname{ker}}

\renewcommand{\deg}{\operatorname{deg}}

\newcommand{\im}{\textup{Im}}

\newcommand{\GA}{\operatorname{GA}}

\renewcommand{\Aff}{\operatorname{Aff}}

\newcommand{\kar}{\textup{char}}
\newcommand{\lcm}{\operatorname{lcm}}
\newcommand{\ord}{\operatorname{ord}}
\newcommand{\inv}{\operatorname{inv}}
\newcommand{\BA}{\operatorname{BAs}}
\newcommand{\lt}{\operatorname{lt}}
\newcommand{\degg}{{\operatorname{deg_L}}}
\newcommand{\BAA}{\operatorname{BA}}
\newcommand{\dlex}{\operatorname{dlex}}

\title{Invariants and conjugacy classes of triangular polynomial maps}
\author{
Stefan Maubach\\ \small Jacobs University Bremen\\
\small
Bremen, Germany \\ \small s.maubach@jacobs-university.de\\}

\begin{document}

\maketitle

\begin{abstract} In this article, we classify invariants and conjugacy classes of triangular polynomial maps. 
We make these classifications in dimension 2 over domains containing $\Q$, dimension 2 over fields of characteristic $p$, and dimension 3 over fields of characteristic zero.
We discuss the generic characteristic 0 case. We determine the invariants and conjugacy classes of strictly triangular maps of maximal order in all dimensions over fields of characteristic $p$. They turn out to be equivalent to a map of the form
$(x_1+f_1,\ldots,x_n+f_n)$ where $f_i\in x_n^{p-1}k[x_{i+1}^p,\ldots,x_n^p]$ if $1\leq i\leq n-1$ and $f_n\in k^*$. 
\end{abstract}

AMS classification:

\section{Introduction}

\label{Sec1}

\subsection{Background}

(For notations and some definitions, please read the next section.)
Triangular polynomial maps are an important class of maps: they are the first nonlinear (nonaffine) polynomial automorphisms one comes up to, and they are a basic building block of many polynomial automorphisms. For one, in dimension two, all automorphisms are compositions of affine and triangular ones. Second, almost all basic examples (like Nagata's automorphism, exponents of locally nilpotent derivatons) are ``almost triangular'' (they are triangular over their invariant ring, or an exponent of a locally nilpotent derivation which is equivalent to a triangular derivation). 

Due to polynomial automorphisms and endomorphisms in general being quite difficult, triangular polynomial maps are often considered trivial. (For example - it's completely trivial to prove the Jacobian Conjecture for triangular polynomial endomorphisms\ldots)
This is deceptive, however: if it is trivial to see that a polynomial is an automorphism, doesn't make it easier to, for example, iterate it, or to find its invariants, or to find its conjugacy class. For all these last questions, there are some reasonably satisfactory answers one can give over fields of characteristic zero, or even rings or domains containing $\Q$ (see section 2). Over fields of characteristic $p$ this becomes much harder already.
It's exactly this characteristic $p$ case, especially the finite field case, which has gained more of an interest, also outside of the field of affine algebraic geometry \cite{Ostafe, Shparlinski, Mau12}.

The overview of this paper is as follows: In section \ref{Sec1} we give background, introduction, definitions etc.
In section \ref{Sec2} we elaborate on the characteristic $0$ and rings-containing-$\Q$ case. We give a link between locally nilpotent derivations, which is reasonably well-known for the invariant case but not that well-known for the image case. 
We determine conjugacy classes in dimension 2 over general rings and in dimension 3 over fields. 
In section \ref{Sec3} we do in all dimensions the equivalent  of the ``locally nilpotent dervation having a slice'' -case for characteristic $p$. 
Since there is no locally nilpotent derivation (or its characteristic $p$ version, a locally iterative higher derivation), this case is truly different, and has a nontrivial answer (whereas the characteristic $0$ case yields ``equivalent to an affine map''). 
We provide a reasonable description of invariants, image and conjugacy classes for this case. This is perhaps the strongest new result of this paper.
In section \ref{Sec4} we determine the dimension 2 case over fields of characteristic $p$. 
In section \ref{Sec5} we briefly discuss automorphisms of finite order, and in section \ref{Sec6} we give further research and acknowledgements.

\tableofcontents
\setcounter{tocdepth}{3}

\subsection{Some notations and basic definitions}

If $R$ is a ring, we will denote $R[x_1,\ldots, x_n]$ as $R^{[n]}$. All rings in this article will be commutative with 1, and most of the time will be domains. We will reserve $k$ for a field.
We define $\GA_n(R)$ as the set of polynomial automorphisms of $R^{[n]}$, and elements 
$F\in \GA_n(R)$ as $F=(F_1,\ldots, F_n)$ where $F_i\in k^{[n]}$. 
$\BAA_n(R)$ is the set of triangular polynomial automorphisms , i.e. where $F_i\in k[x_i,x_{i+1},\ldots, x_n]$. (BA stands for Borel Automorphisms, see \cite{hanoi}.)
It follows that $F_i=a_ix_i+f_i$ where $f_i\in k[x_{i+1},\ldots, x_n]$. 
The group $\BA_n(R)$ is the set of strictly upper triangular polynomial maps, i.e. maps of the form $F=(x_1+f_1,\ldots, x_n+f_n)$ where
$f_i\in k[x_{i+1},\ldots, x_n]$. $\Aff_n(R)$ is the set of affine maps, i.e. compositions of linear maps and translations.

\subsection{Unipotent and triangular maps}
\label{Sec2}

\begin{definition} Let $F\in \GA_n(R)$. Then $F$ is called {\em locally finite} (short LF) if  there exist $d\in \N$ and $a_i\in R$ such that $F^d=\sum_{i=0}^{d-1} a_i F^i$. It follows that $\deg(F^m)$ is bounded. 
In case the polynomial $T^d-\sum_{i=0}^{d-1} a_iT^i =(T-1)^d$, then we say that $F$ is unipotent. 
\end{definition}

(Note: in some articles, LF is called ``algebraic'', see for example \cite{StKr}.)

\begin{example} All elements in $\BAA_n(R)$ are locally finite. The elements in $\BA_n(R)$ are unipotent.
\end{example}

It will be convenient to abbreviate elements in $\BA_n(k)$ which have many identity components, for example
\[ (x_1,\ldots, x_{i-1}, x_i+f_i, x_{i+1},\ldots, x_n)=(x_i+f_i)\]
\[ (x_1,\ldots, x_{i-1}, x_i+f_i, x_{i+1},\ldots,x_{j-1}, x_j+f_j, x_{j+1},\ldots,  x_n)=(x_i+f_i, x_j+f_j)\]
etc. 

If $A\subseteq R^{[n]}$ then we denote $A^F=\{a\in A ~|~F(a)=a\}$. In case $A$ is clear (mostly, meaning $A=R^{[n]}$) then we write $\inv(F)=A^F$. 

In this article, our goal is to understand elements in $\BA_n(k)$ for any field $k$. In particular, we want to understand the following:
\begin{itemize}
\item What are the invariants of some $F\in \BA_n(k)$?
\item What are the conjugacy classes of $\BA_n(k)$ in $\BA_n(k)$?
\item What are the conjugacy classes of $\BA_n(k)$ in $\BAA_n(k)$?
\item What are the conjugacy classes of $\BAA_n(k)$ in $\BAA_n(k)$?
\end{itemize}
These questions will be a bit too ambitious to solve in general - in fact, one can say that even in characteristic zero, the invariants are quite complicated. (See example \ref{Roberts}.)
For completeness sake, we will first discuss the characteristic zero case, after which we will discuss the characteristic $p$ case, which will be more involved. 
We will also consider the above questions over rings (domains), as these sometimes can help us answer the question over fields in one variable higher.

\section{Characteristic zero}

Let $R$ be a domain of characteristic zero (i.e. $\Q\subseteq R$). 

\begin{definition} Let $D:R^{[n]}\lp R^{[n]}$ be an $R$-linear map. Then $D$ is called a derivation if $D(fg)=fD(g)+D(f)g$ for all $f,g\in R^{[n]}$. $D$ is called locally nilpotent if for every $f\in R^{[n]}$, there exists $d\in \N$ such that $D^d(f)=0$. 
$D$ is called triangular if $D(x_i)\in R[x_{i+1},\ldots, x_n]$. 
A slice of $D$ is an element $s\in R^{[n]}$ such that $D(s)=1$. 
\end{definition}

\begin{lemma} \label{lnd} Let $R$ be a domain containing $\Q$. Then $F\in \GA_n(R)$ being unipotent is equivalent to $F=\exp(D)$ for some locally nilpotent derivation $D$. 
\end{lemma}

\begin{proof} In \cite{Fu-Mau} lemma 2.3, the above theorem is proven for the case $R=k$, a field. 
If we let $k$ be the quotient field of $R$, we thus can find some l.n.d. $D$ having coefficents in $k$. 
We will show that $D$ actually has coefficients in $R$. 
We know that for each $m\in \N$, we get $F^m(x_i)=\exp(mD)(mx_i)\in R^{[n]}$. Let $d\in \N$ such that 
 $D^d(x_i)=0$, and let $V=\sum_{m\in \N} R \exp(D)(mx_i) $.
We claim that $D(x_i)\in V$. 
Indeed:
\[ 
\begin{pmatrix}
\exp(0D)(x_i)\\
\exp(D)(x_i)\\
\vdots\\
\exp((d-1)D)(x_i)
\end{pmatrix}
=
M\begin{pmatrix} x_i\\ D(x_i)\\ \vdots \\D^{d-1}(x_i)
\end{pmatrix}\]
where $M$ is some Vandermonde matrix. So indeed $D^j(x_i)\in V$ for each $j$. Since $V\subseteq R^{[n]}$ 
we are done.
\end{proof}

This fact makes the characteristic zero case so understandable. Before we state our main theorem, let us elaborate a bit. 
We have the following well-known theorem:

\begin{lemma} \label{lnd2} Let $D$ be a locally nilpotent derivation on a ring $A$ having a slice $s$. Then 
$A^D=\ker(D)=k[a_1,\ldots, a_n]$ for some $a_i\in A$, and $A=A^D[s]$. \\
Furthermore, $\im(D)=A$ if and only if $D$ has a slice.
\end{lemma}

A locally nilpotent derivation having a slice is obviously something very useful for such a derivation, and correspondingly for the map. 
It is conjectured that such a slice $s$ must automatically be a coordinate  (meaning
there exist mates $s_2,\ldots, s_n$ such that $R[s,s_2,\ldots, s_n]=R^{[n]}$). Note that $\exp(D)(s)=s+1$, making the following definition natural:

\begin{definition} We say that $F\in \GA_n(R)$ has a slice if there exists $s\in  R^{[n]}$ such that $F(s)=s+1$. 
If $s$ is a coordinate, we say that $s$ is a coordinate slice. 
\end{definition}

When conjugating, we encounter the following fenomenon: $(x_{i}+g_{i} )(x_i+f_i, \vec{F}_{i+1})(x_{i}-g_{i})
=(x_i+f_i+g_i(\vec{F}_{i+1})-g_i, \vec{F}_{i+1})$. This gives rise to the following definitions:

\begin{definition}
Given $F\in \BA_n(R)$, define $N:=F-I$, $\vec{F}_i=(F_i,F_{i+1},\ldots, F_n)$, $N_i=\vec{F}_i-(x_i,\ldots, x_n)$ (i.e. $N_1=N$). 
\end{definition}

\begin{lemma} \label{conj}  $(x_i+g_i) F (x_i-g_i) = (F_1(x_i-g_i), \ldots, F_{i-1}(x_i-g_i), F_i+N(g_i), F_{i+1},\ldots, F_n)$. 
\end{lemma}

The above (trivial) lemma explains how we can determine equivalence classes of elements in $\BA_n(R)$: we first conjugate by
a map $(x_n+g_n)$ to bring $f_n$ to a standard form, then conjugate by $(x_{n-1}+g_{n-1})$ to bring $f_{n-1}$ to a standard form, etc.
This means that we need to understand $\im(N)$. (Incidentally, $\ker(N)=\inv(F)$, and thus forms a similar role as a locally nilpotent derivation!)

We will prove the following theorem:

\begin{theorem} \label{ThInvIm0} Let $R$ be a ring containing $\Q$. Assume $F=\exp(D)$ is unipotent on $R^{[n]}$. Then\\
(1) $\inv(F)=\ker(D)$ $(=\ker(N))$,\\
(2) $\im(N)=\im(D)$.
\end{theorem}

\begin{proof} 
$\inv(F)=\{f~|~F(f)=f\}=\{f ~|~(F-I)(f)=0\}=\ker(N)$.
Now $\inv(F)$ is the invariants of $\exp(D)$ which is well-known to be equal to $\ker(D)$ if $D$ is locally nilpotent, so the first statement holds.

Define a degree function on $R$ by $r\not =0$ then $\deg(r)=\max\{d ~|~D^d(r)\not =0\}$, $\deg(0)=-\infty$. (It is well-known that this yields a degree function if $D$ is locally nilpotent.) Denote $R_d:=\{r\in R~|~\deg(R)\leq d\}$.
Given $f\in R$, we proceed by induction to $d=\deg(f)$ to prove $f\in \im(D)\desda f\in \im(N)$. \\
$d=-\infty:$ then $f=0$ and the statement is true.\\
Assume that  $\im(D)\cap R_{d-1}= \im(N)\cap R_{d-1}$, let $f\in R_d$. 

Assume $f\in \im(D)$, then $f=D(g)$ for some $g\in R$. Now $N(g)=D(g)+h$ where $h=\sum_{i=2}^{\infty} \frac{1}{i!}D^i(g)\in \im(D)$.
Since $\deg(h)=\deg(g)-2=d-1$ we use induction and find $h'$ such that $N(h') =h$, then $N(g-h')=D(g)=f$, and thus $f\in \im(N)$.

Assume $f\in \im(N)$, then $f=N(g)=D(g)+h$ where $h=$ as above. Since $h\in \im(D)$ we find $h'$ such that $D(h')=h$. 
Thus $D(g+h')=D(g)+D(h')=D(g)+h=f$ and thus $f\in \im(D)$.
\end{proof}

In some sense, the above theorem only translates the problem. $\im(D)$ and $\ker(D)$ are not really easy even for triangular derivations. One example:

\begin{example} \label{Roberts} Let $F=(x_1+x_5^3, x_2+x_6^3, x_3+x_7^3, x_4+(x_5x_6x_7)^2, x_5,x_6,x_7)\in \BA_n(k)$, where $\kar(k)=0$.
Then $\inv(F)$ is not finitely generated. 
\end{example}

The above example is nothing other than the exponent of Robert's example \cite{Rob}, a locally nilpotent derivation whose kernel is not finitely generated. Note there exist counterexamples by Freudenburg \& Daigle-Freudenburg in dimensions 5 and 6 too \cite{F, DF}.

\subsection{Conjugacy classes}

We now want to give some answer to how to describe (representants) of conjugacy classes. 
The generic case is rather complicated, and we will not fully answer it (similarly as no one truly can answer exactly what $\im(D)$ and $\ker(D)$ are in general).  We will focus on some special cases, especially as we want to determine 
what happens in low dimensions.
A first case is easy, but it is an important case:

\begin{proposition} \label{CC1} If $F=(x_1+f_1,\ldots, x_n+f_n) \in \BA_n(R)$ where $f_n\in R^*$, then 
$F$ is in the same conjugacy class as $(x_n+f_n)$. 
\end{proposition}

\begin{proof} Since $\im(N_i)=k[x_i,\ldots,x_n]$ (as $F=\exp(D)$ where $D$ has a slice, and using theorem \ref{ThInvIm0}) for each integer $2\leq i\leq n$, we can find $g_i\in k[x_i,\ldots,x_n]$ such that 
$N(g_i)=N_i(g_i)=f_i$. This means that we can conjugate any $F=(x_1+f_1,\ldots, x_j+f_j,x_{j+1},\ldots,x_{n-1},x_n+f_n)$ by $(x_j+g_j)$ (see lemma \ref{conj})
and get a map of the form $(x_1+f_1',\ldots, x_{j-1}+f_{j-1}', x_j, \ldots,x_{n-1},x_n+f_n)$. Continuing this process we end up at the map $(x_n+f_n)$. 
\end{proof}

The above proof contains a little bit what one can do in general, and what is explained in lemma \ref{conj}: given $F=\exp(D)=(x_1+f_1,\ldots,x_n+f_1)$
we first conjugate by $(x_{n-1}+g_{n-1})$ for some appropriate $g_{n-1}$ (conjugating by $(x_n+g_n)$ changes nothing). 
This changes the $n-1$-term into $x_n+f_{n-1}+N(g_{n-1})$, and this requires us to understand $N(R[x_n])$, which is equal to 
$D(R[x_n])$. We then can pick a representant of $f_{n-1}$ modulo $\im(D)$ and continue by conjugating by $(x_{n-2}+g_{n-2})$ etc. 
It is therefore important to understand $R^{[n]}/\im(D)$ for a triangular $D$; it enables one to understand the conjugacy classes in $\BA_n(R)$. 

At first it seems like $R^{[n]}/\im(D)$ might be understandable, allow us to elaborate: $\im(D)$ is a free $R$-module generated by $\{D(T)~|~T$ is a monomial in $R^{[n]}\}$. 
So you can reduce each given $g\in k^{[n]}$ modulo the highest degree terms appearing in these $D(T)$ (with respect to a lexicographic grading $\degg$ given by $x_1>>x_2>>\ldots>>x_n$), and get something unique. (Indeed, with respect to this lexicographic ordering, denoting $\lt(g)$ as the leading term of $g$, one can give a nice description of $\lt(D(g))$ related to $\lt(g)$ with respect to this grading.) 

However, this does not necessarily give unique representants of $R^{[n]}/\im(D)$.
The main reason why this fails is that there can exist polynomials $g,h$ with the property that $\degg(g)>\degg(h)$ but $\degg(D(g))<\degg(D(h))$. We give the following example:

\begin{example} Let $D=x_2\partial_1+\partial_2$. Then $\lt(x_1)> \lt(x_2^2)$ but $\lt(D(x_1))<\degg(D(x^2_2)$.
\end{example}

Nevertheless, this is an important idea to keep in mind in the low dimensional cases we shall now consider.

\subsection{Conjugacy classes within $\BA$ in dimensions 2 and 3}

\begin{theorem} \label{ThCC2} The conjugacy classes of $\BA_2(R)$ in $\BA_2(R)$ where $\Q\subseteq R$ is a domain, are parametrized by pairs
$(\bar{f}_1,f_2)$ where
\begin{enumerate}
\item $f_2=0$, $\bar{f_1}=a_nx_2^n+a_{n-1}x_2^{n-1}+\ldots+a_0$ where $a_{n-1}$ is picked as a unique representant in $R/a_nR$,
\item  $f_2\in R\backslash\{0\}$, $ \bar{f_1}=\bar{a}_nx_2^n+\bar{a}_{n-1}x^{n-1} +\ldots+\bar{a}_0 \in R/(f_2) [x_2]$, where $\bar{a}_{n-1}$ is a uniquely picked representant from $(R/(f_2))/(\bar{a}_n)$. 
\end{enumerate}
In particular, if $R=k$ a field, then the classes are
\begin{enumerate}
\item  $(x_1+f(x_2), x_2)$, $f(x_2)=f_nx_2^n+f_{n-2}x_2^{n-1}+\ldots+f_0 \in k[x_2]$ where $f_n\not = 0$ (i.e. the next-to highest term has coefficient zero),
\item $(x_1,x_2+\lambda)$, $\lambda\in k$.
\end{enumerate}
\end{theorem}

\begin{proof} 
Let $F=(x_1+f_1,x_2+f_2)$ be the triangular map. The $x_2+f_2$ part cannot be changed. We will first conjugate by something of the form $(x_1+g_1)$ and then of the form $(x_2+g_2)$. $F$ equals $\exp(D)$ where $D=f_1'\partial_1+f_2\partial_2$ and some $f_1'\in R[x_2]$.
Since conjugation by $(x_1+g_1)$ means modifying $f_2$ by $N(R[x_2])$, we need to understand $R[x_2]/N(R[x_2])$, and using theorem \ref{ThInvIm0} part (2) we see $N(R[x_2])=D(R[x_2])=\im(f_2\partial_2)=f_2R[x_2]$. We can thus conclude that $(x_1+g_1,x_2+g_2)$ and $F$
are equivalent under conjugation by some $(x_2+h_2)$ 
if and only if $f_2=g_2$ and $g_1\in f_1+f_2R[x_2]$.
Let us assume $f_1=a_nx^n+a_{n-1}x^{n-1}+\ldots+a_0$. We can still conjugate by $(x_2+\tilde{h}_2)$, which means that we can change the term of degree $n-1$: $f_1(x_2+\tilde{h}_2)$ has the top part $a_nx^n+(n\tilde{h}_2a_n+a_{n-1})x^{n-1}+\ldots$. Thus we can change the $n-1$ term
by any element in $a_nR$. 
Note that conjugating by any $(x_1+\tilde{h}_1)$ or $(x_2+\tilde{h}_2)$ disturbes the standard form (it is important that $R$ is a domain here!). 
This finally proves the theorem for rings $R$.

The second case $R=k$ follows directly from the previous, but can also be partially proven using lemma \ref{CC1}.
\end{proof}

\begin{theorem} \label{ThCC3} The conjugacy classes of $\BA_3(k)$ in $\BA_3(k)$ are
\begin{enumerate}
\item $(x_1,x_2,x_3+\lambda)$, $\lambda\in k$,
\item  $(x_1+f_1(x_2,x_3),x_2+f_2(x_3), x_3)$ where 
$f(x_2)=f_nx_2^n+f_{n-2}x_2^{n-1}+\ldots+f_0 \in k[x_2]$ where $f_n\not = 0$,
\[ f_1(x_2,x_3)=a_n(x_3)x_2^n+a_{n-1}(x_3)x_2^{n-1}+\ldots+a_0(x_3), \]
and for $i=0,1,\ldots, n-2, n$, $a_i(x_3)$ is the lowest degree element in $a_i(x_3)+f_2(x_3)k[x_3]$, 
while $a_{n-1}(x_3)$ is the lowest degree element in $a_{n-1}(x_3)+(f_2(x_3), a_n(x_3))k[x_3]$.
\end{enumerate}
\end{theorem}

\begin{proof} Using proposition \ref{CC1} we see that if $f_3\in k^*$, then the map is equivalent to 
$(x_1,x_2,x_3+f_3)$, yielding the first case.
Left is $f_3=0$, which comes down to the general case of theorem \ref{ThCC2}, picking $R=k[x_3]$. 
The result now immediately  follows, keeping in mind that  since lowest degree elements in sets like $a_i(x_3)+f_2(x_3)k[x_3]$ are unique: if $a_i\in f_2k[x_3]$, then 
the lowest element is zero.
\end{proof}

\subsection{Conjugacy classes of $\BA$ in $\BAA$ in dimensions 2 and 3}

Any element in $\BAA_n$ can be written as $DG$ where $D$ is diagonal linear, and $G\in \BA_n$. 
This means that if we try to determine a representant of a conjugacy class of $F\in \BA_n$, we conjugate by $
DG$, i.e. consider $DGFG^{-1}D^{-1}$. We thus first pick a representant in $GFG^{-1}\in \BA_n$, and on top of that conjugate by a diagonal linear map.

Conjugation by $(\lambda x_3)$ where $\lambda\in R^*$ on $R[x_3]$ gives a group action
$R^*\times R[x_3]\lp R[x_3]$, given by $\lambda \cdot f(x_3)= \lambda^{-1}f(\lambda x_3)$. 
This case thus gives an additional gathering of conjugacy classes under these kind of orbits. There's not really a simplification of this possible, unless $R$ is an algebraically closed field, or the reals or something specific. And even then it is limited: 
the polynomials of the form $ax^m+bx^{l}$ can be conjgated to the form $x^m+\tilde{b}x^l$, and then the coefficient $\tilde{b}$ can 
be changed if $l$ does not divide $m$ by some conjugation by $\lambda x$ where $\lambda$ is an $m$-th root of unity, etc\ldots

\begin{theorem} \label{ThCC2a} The conjugacy classes of $\BA_2(R)$ in $\BAA_2(R)$ where $\Q\subseteq R$ is a domain, are parametrized by pairs
$(\bar{f}_1,f_2)$ where
\begin{enumerate}
\item $f_2=0$, $f_1=a_nx_2^n+a_{n-1}x_2^{n-1}+\ldots+a_0$ where $a_{n-1}$ is picked as a unique representant in $R/a_nR$,
and then additionally $f_1$ is a unique element in the orbit of the action $\lambda \lp \lambda^{-1}f_1(\lambda x_2)$.
\item  $f_2\in R\backslash\{0\}$, $ \bar{f_1}=\bar{a}_nx_2^n+\bar{a}_{n-1}x^{n-1} +\ldots+\bar{a}_0 \in R/(f_2) [x_2]$, where $\bar{a}_{n-1}$ is a uniquely picked representant from $(R/(f_2))/(\bar{a}_n)$. 
Then, additionally $f_1$ is a unique element in the orbit of the action $\lambda \lp \lambda^{-1}f_1(\lambda x_2)$.
\end{enumerate}
In particular, if $R=k$ a field, then the classes are
\begin{enumerate}
\item  $(x_1+f(x_2), x_2)$, $f(x_2)=f_nx_2^n+f_{n-2}x_2^{n-1}+\ldots+f_0 \in k[x_2]$ where $f_n\not = 0$ (i.e. the next-to highest term has coefficient zero). Additionally $f_1$ is a unique element in the orbit of the action $\lambda \lp \lambda^{-1}f_1(\lambda x_2)$.
\item $(x_1,x_2+1)$, 
\item $(x_1,x_2)$.
\end{enumerate}
\end{theorem}

\begin{proof} 
The form $(x_2+f_2)$ can be conjugated by $(x_1,\lambda x_2)$ where $\lambda \in R^*$.
This explains why we pick  $\bar{f}_2\in R/R^*$. Let $f_2$ be any representant of $\bar{f}_2$. 
Then the ideal $Rf_2$ is always the same ideal, regardless of representant. 
Now $(x_1+f_1,x_2+f_2)$ can be conjugated by $(\lambda x_1,x_2)$ to get any element in $R^*f_1$ in stead of $f_1$. 
This proves the first statement.

The case $R=k$ is now trivial.
\end{proof}

\begin{theorem} \label{ThCC3a} The conjugacy classes of $\BA_3(k)$ in $\BAA_3(k)$ are
\begin{enumerate}
\item $(x_1,x_2,x_3)$,
\item $(x_1,x_2,x_3+1)$, 
\item  $(x_1+f_1(x_2,x_3),x_2+f_2(x_3), x_3)$ where 
$f(x_2)=f_nx_2^n+f_{n-2}x_2^{n-1}+\ldots+f_0 \in k[x_2]$ where $f_n\not = 0$.
 Additionally, $f_2$ is a unique element in the orbit of the action $\lambda \lp \lambda^{-1}f_2(\lambda x_2)$. Now
\[ f_1(x_2,x_3)=a_n(x_3)x_2^n+a_{n-1}(x_3)x_2^{n-1}+\ldots+a_0(x_3), \]
and for $i=0,1,\ldots, n-2, n$, $a_i(x_3)$ is the lowest degree element in $a_i(x_3)+f_2(x_3)k[x_3]$, 
while $a_{n-1}(x_3)$ is the lowest degree element in $a_{n-1}(x_3)+(f_2(x_3), a_n(x_3))k[x_3]$.
Furthermore, the sequence $(a_n, a_{n-1},\ldots,a_0)$ is picked uniquely from the orbit under conjugation by $(\lambda x_3)$.
\end{enumerate}
\end{theorem}

\begin{proof} Using proposition \ref{CC1} we see that if $f_3\in k^*$, then the map is equivalent to 
$(x_1,x_2,x_3+f_3)$, yielding the first case.
Left is $f_3=0$, which comes down to the general case of theorem \ref{ThCC2}, picking $R=k[x_3]$. 
Picking a representant in $f_1+(k[x_3]/f_3)[x_2]$ of lowest degree is unique, and the theorem is proven.
\end{proof}

\subsection{Conjugacy classes of $\BAA$ in itself}

The 2-variable case over a domain (and with that, the 3-variable case over a field) are much more involved (but doable in future research). 
Here we aim at the 2-variable case over a field $k$ of characteristic zero. We start with the one-variable case over a ring:

\begin{lemma} \label{G1} Let $R$ be a reduced ring (not necessarily containing $\Q$). Then the conjugacy classes of $\BAA_1(R)=\GA_1(R)$  are
\begin{enumerate}
\item $x+b$ where $b$ is a (unique) representant in $R$ of  $R/R^*$ (the orbit space of the action $R^*\times R\lp R$),
\item $ax+b$ where $a\not =1$, and $b$ is a unique representant in $R$ of $(R/(a-1)R)/R^*$  (the orbit space of the action $R^*\times R/(a-1)R\lp R/(a-1)R$).
\end{enumerate}
In particular, if $R=k$ is a field, then the conjugacy classes are $ax$, $a\in k*$ and $x+1$. 
\end{lemma}

\begin{proof}
A generic element looks like $ax+b$ where $a\in R^*, b\in R$. If $a=1$, then conjugating by $\lambda x$ can change $b$ to $\lambda b$,
meaning that we have to pick $b$'s uniquely from each orbit of the natural action $R^*\times R\lp R$.  Conjugation by $x+\lambda$ does nothing, so this gives the first case.

Now if $a\not =1$ then let us conjugate by a generic element $(\mu x-\lambda)$ where $\mu\in R^*, \lambda\in R$.
Then 
$(\mu x-\lambda) (ax+b)( \mu^{-1}x+\mu^{-1}\lambda)=(ax+\mu b+ \lambda (a-1))$. 
 This means that we cannot change $a$, but we can change $b$ to 
any element in $R^* b+ R(a-1)$. 
\end{proof}

(The above can be extended easily to $R$ not a domain, but in case  then one should be careful with the definition of $\BAA$: does one mean all invertible triangular maps, or maps which send each variable $x_i$ to $\lambda_i x_i +f_i(x_{i+1},\ldots,x_n)$. i.e. is 
$\BAA_1(R)$ polynomials of degree 1 or $\BAA_1(R)=\GA_1(R)$?)

\begin{lemma} Let $k$ be a field of characteristic zero. The conjugacy classes of $\BAA_2(k)$ are 
\begin{enumerate}
\item Second component $y$:
\begin{enumerate}
\item $(x,y)$,
\item $(x+f(y),y)$ where $f(y)=y^d+a_{d-2}y^{d-2}+\ldots a_0$ (i.e. monic and second coefficient zero),
and picked uniquely from the set $\{ f(cy)~|~c^{d-1}=1\}$. 
\item $(bx,y)$ where $b\not =0,1$,
\end{enumerate}

\item  Second component $y+1$:
\begin{enumerate}
\item $(x,y+1)$
\item $(ax,y+1)$ where $a\in k^*, a\not =1$
\end{enumerate}

\item Second component $ay$ where $a\not =1,0$:
\begin{enumerate}
\item $(bx,ay)$ if there is no $m\in \N$ such that $b^m=a$, 
\item $(a^mx+y^m,ay)$ if $a$ is no root of unity,
\item $(a^mx+y^mf(y^r),ay)$ if $r=\ord(a)$, and $f$ monic.
Furthermore, $y^mf(y^r)$ is uniquely picked from $\{(\mu y)^mf(\mu^r y^r) ~|~ \mu^{m+rd}=1\}$ where $d=\deg(f)$. 
\end{enumerate}

\end{enumerate}

Another classification is:
\begin{itemize}
\item[\bf A] (affine), and then
\begin{enumerate}
\item $(bx,y+c)$ where $b\in k^*$, $c\in \{0,1\}$,
\item $(bx,ay)$  where $a,b\in k^*$,
\end{enumerate}
\item[\bf S] (sequential)\\
$(a^mx+y^mf(y^r), ay)$ where $m\in \N$,  $r=\ord(a)$ ($r=0$ if $a$ is no root of unity), and $y^mf(y^r)$ is monic.  
Furthermore, $y^mf(y^r)$ is uniquely picked from $\{(\mu y)^mf(\mu^r y^r) ~|~ \mu^{m+rd}=1\}$ where $d=\deg(f)$. 
\end{itemize}
\end{lemma}

\begin{proof} 
Write $F=(bx+f(y), ax+\lambda)$. 
Using lemma \ref{G1}, and the fact that $a-1$ is invertible if $a\not =1$ and thus $(a-1)R=(a-1)k=k$,
 the second component is one of three $y, y+1, ay$ where $a\not =1,0$. We will consider these three cases.

\underline{\bf Case $a=1, \lambda=0$ }

Then we have $(bx+f(y), y)$. We can actually apply lemma \ref{G1} to $(bx+f(y))$ on $R=k[y]$. 
If $b\not =1$ then $k[y]=(b-1)k[y]$,  so we can get $f=0$. 
In case $b=1$, $f$ is unique in $k[y]/k^*$ - which means we can pick $f$ monic (or $f=0$).
Now we've only taken into account conjugations by $(cx+g(y))$, however -
we still need to check what happens under conjugation by $(cy+g)$, i.e. conjugations by $(cy)$ and $(y+g)$. 
Write $d=\deg(f)$.
Using conjugation by $(y+g)$ where $g\in k$, we can make sure that the $d-1$ coefficient is zero. 
Conjugation by $cy$ can change $f(y)$ by a nonzero scalar. However, $f(y)$ needs to stay monic, so we can only 
change $f(y)$ into $c^{-1}f(cy)=c^{d-1}y^d+\ldots$ where $c^{d-1}=1$.

\underline{\bf Case $a=1, \lambda=1$ }

Conjugating by $(x+g(y))$ gives $(bx+bg(y)-g(y+1)+f(y),y+1)$ which means we need to consider the map $k[y]\lp k[y]$ given 
by $g(y)\lp bg(y)-g(y+1)$. This map is surjective in all cases:
if $b\not =1$ then  $y^m$ is mapped to a polynomial of degree $m$ and the map is bijective, actually.
If $b=1$ then $y^m$ is mapped to a polynomial of degree $m-1$, and the map is still surjective (though not bijective). 
Thus, we conjugate to $(bx,y+1)$. Any conjugation will disturb this form or leave it unchanged, so this is the final form for this case.

\underline{\bf Case $a\not =1,0, \lambda=0$}

Conjugation by $(x+g(y))$ yields $(bx+bg(y)-g(ay)+f(y), ay)$. This means we need to understand the map $k[y]\lp k[y]$ 
given by $g(y)\lp bg(y)-g(ay)$. This map decomposes into homogeneous parts, so we need to consider
$y^m\lp by^m-a^my^m$. This map is surjective if $b-a^m\not =0$ for all $m$. 
If $b=a^m$, then the $y^m$ part cannot removed. So let $m$ is the lowest integer such that $b=a^m$. If there's another integer $m'$ such that $b=a^{m'}$, then $a^{m'-m}=1$ and thus $a$ has finite order. Concluding, we get 
$(a^mx+cy^m, ay)$ where $c\in k$ if $a$ is no root of unity, and $(a^mx+y^mg(y^r), ay)$ where $r=\ord(a)$, and $m<r$. 
Conjugating by $(\mu x)$ we can make sure that $c=1$ (or $c=0$) and $g(y^r)$ monic. 
Conjugating by $(\mu y)$ we change the monicness of $y^mg(y^r)$ unless $\mu^d=1$ where $d=\deg(y^mg(y^r))$.
\end{proof}

\subsection{Higher dimensions}

There are some higher dimensional cases which we expect that can be aquired by some more effort (but become rather technical): $\BA_3(R)$ in $\BA_3(R)$ and $\BAA_3(R)$ for domains $R\supset \Q$, and with that also $\BA_4(k)$ in $\BA_4(k)$ and $\BAA_4(k)$ for fields $k$ of characteristic zero. Also
$\BAA_2(R)$ in $\BAA_2(R)$ and with that also $\BAA_3(k)$ in $\BAA_3(k)$ should be achievable.
It is a bit of a  challenge to give a good description which doesn't ``explode'', however.

We do expect that (perhaps in dimension 5 or 6) it is very, very hard or impossible to truly classify the conjugacy classes. We expect similar difficulties as with $\ker(D)$, which can be infinitely generated in dimension 5, and where it's unknown if it can be infinitely generated in dimension 4.

\section{Characteristic $p$: strictly triangular maps of maximal order}

\label{drie}
\label{Sec3}

\subsection{Introduction}

In this whole section \ref{drie}, $k$ is a field of characteristic $p$. 
The characteristic $p$ case brings in additional difficulties with respect to the goal of classifying conjugates. 
In characteristic zero we have proposition \ref{CC1}, which essentially states that if $F$ has last component $x_n+f_n$ where $f_n\in k^*$, then your map is very simply up to a conjugation. Another issue is that in characteristic $p$, we have no true equivalent of 
lemma \ref{lnd}, which states that $F=\exp(D)$ for some locally nilpotent derivation $D$. 
The equivalent object in characteristic $p$ to a locally nilpotent derivation is a locally finite higher iterative derivation (see \cite{hd}, we will not give details in this article),
which has the following issue: if $F=\exp(D)$ where $D$ is such a  locally finite higher iterative derivation, then $F^p=I$. 
This means, that it doesn't even include all strictly triangular polynomial maps, let alone all unipotent maps. 
Hence, we need to resort to (slightly) different methods, and will have difficulty going into dimension 3 (and higher) except for special cases. 

\begin{definition} If $F\in \BA_n(R)$, define
$M_i:=\sum_{j=0}^{p^{n+1-i}-1} {F}^j$ for $1\leq i\leq n$. We recall the definition $N=F-I$. We say $M=M_1$. 
\end{definition}

\begin{lemma} Let $R$ be a commutative ring of characteristic $p$. Let $F=(x_1+f_1,\ldots, x_n+f_n)\in \BA_n(R)$. \\
(1) $F^p$ fixes $x_n$, i.e. $F^p\in \BA_{n-1}(R[x_n])$. \\
(2)  $F^{p^n}=I$. \\
(3) 
$F^{p^m}=(x_1+g_1,\ldots, x_{n-m}+g_{n-m}, x_{n-m+1}, \ldots, x_n)$ for some $g_i\in R[x_{i+1},\ldots,x_n]$, 
and $g_{n-m}=M_{n-m}(f_{n-m})\in \im(M_{n-m})$. 
\end{lemma}

\begin{proof}
(1) is trivial. (2) follows from (1) using induction. (3) we prove by  induction:  we prove that if $1\leq d\leq  p^{m}$ then the $n-m$-th part of $F^d$  is $x_{n-1}+\sum_{i=0}^{d-1} F^i(f_{n-m})$. 
Indeed, for $d=1$ this is correct. Now assume $d$. Then the $n-m$-th component of  $F^{d+1}=F^d\circ F$ equals
 $(x_{n-m}+\sum_{i=0}^{d-1} F^i(f_{n-m}) ) \circ F =x_{n-m}+f_{n-m}+\sum_{i=1}^{d} F^i(f_{n-m}) )$, proving the induction step.
\end{proof}

\subsection{Main theorems on invariants and an exact sequence}

The important special case we consider in all dimensions is the case where $\ord(F)=p^n$. On this, we want to prove the following two theorems:

\begin{theorem}\label{T1}
Let $F=(x_1+f_1,\ldots, x_n+f_n)\in \BA_n(k)$ where  $\ord(F)=p^n$. 
Then $\inv(F)=k[\tilde{x}_1,\ldots, \tilde{x}_n]$ where 
$\tilde{x}_i=x_i^p-a_i^{p-1}x_i+b_i$ where $a_i,b_i\in k[x_{i+1},\ldots, x_n]$. 
\end{theorem}

\begin{theorem}\label{T2}
Let $F=(x_1+f_1,\ldots, x_n+f_n)$ such that $\ord(F)=p^n$. 
Then the sequence
\[ 0\lp (k^{[n]})^F \lp k^{[n]} \overset{N}{\lp} k^{[n]} \overset{M}{\lp}  (k^{[n]})^F \lp 0\]
and the sequence
\[ 0\lp \im(N) \lp k^{[n]}\overset{M}{\lp} k^{[n]} \overset{N}{\lp} \im(N) \lp 0 \]
are exact.  
\end{theorem}

The proof of both theorems is rather involved,  as the proof of the $n$-dimensional case of any of the theorems involves the $n-1$-dimensional case of both theorems. In fact, if we denote $T_{\ref{T1}}[n]$  by the statement ``theorem \ref{T1} is true in dimension $n$'' 
and similarly $T_{\ref{T2}}[n]$, the proof will follow the following scheme:\\
\begin{itemize}
\item Prove $T_{\ref{T1}} [1]$ and $T_{\ref{T2}}[1]$,
\item Prove ($T_{\ref{T1}}[n-1]$, $T_{\ref{T2}}[n-1]$)$\lp$ $T_{\ref{T1}}[n]$,
\item Prove ($T_{\ref{T1}}[n]$, $T_{\ref{T2}}[n-1]$)$\lp$ $T_{\ref{T2}}[n]$.
\end{itemize}

\subsection{Generalities on linear maps of order $p^n$}

\begin{lemma} \label{SM1}Let $V$ be a $k$-vector space where $k$ is a field of characteristic $p$, and let $L:V\lp V$ be a  $k$-linear map such that $L^{p^m}=I$.
Then\\
(1)  $L=I+N$ where $N^{p^m}=0$,\\
(2) the only eigenvalue of $L$ is 1, \\
(3) $L$ is locally finite,\\
(4) $I+L+L^2+\ldots +L^{p^m-1}=N^{p^m-1}$.
\end{lemma}

\begin{proof}
(1) Write $N=L-I$. Then $L^p=(I+N)^p=I+N^p$ so $L^p=I\desda N^p=0$.\\
(2) follows from all eigenvalues being unit roots of order $p^{m}$, and the fact that the only $p$-th root of 1 in characteristic $p$ is 1 (the only solution to $x^p-1$ is 1). \\
(3)  Let $w\in V$. Then $w,L(w),\ldots, L^{p^m-1}(w)$ spans a finite dimensional subspace $V_w$ of $V$ such that $w\in W$. Hence $L$ is locally finite.\\
(4) 
\[ \sum_{n=0}^{p^m-1} L^n=\sum_{n=0}^{p^m-1} (I+N)^n=
\sum_{n=0}^{p^m-1}\sum_{i=0}^n {n \choose i} N^i=
\sum_{i=0}^{p^m-1} \left(  \sum_{n=i}^{p^m-1} {n \choose i} \right) N^i=\]
\[
\sum_{i=0}^{p^m-1} {p^m\choose i+1} N^i=N^{p^m-1}.\\
\]
\end{proof}

\begin{remark} \label{SM1a} Note that $M:=N^{p^m-1}$.
\end{remark}

\begin{lemma}\label{SM3}
 Let $L$ be as in lemma \ref{SM1}. Define $N=L-I$ and $M=I+L+L^2+\ldots+L^{p^m-1}$.
The sequence
\[ 0\lp V^{L}\lp V\overset{N}{\lp} V\overset{M}{\lp} V^L\lp V^L/\im(M) \lp 0 \]
is a well-defined complex sequence.
The only non-trivial homology is $\ker(M)/\im(N)$. \\
If $V$ is not only a $k$-module but also a ring (i.e. a $k$-algebra), and $L$ is a ring homomorphism of $V$, then $\im(M)$ is an ideal of $V^L$.
\end{lemma}

\begin{proof}
{\em Exact at the first $V^L$} is trivial. \\
{\em Exact at the first $V$} follows since $\ker(N)=\ker(L-I)=\{v\in V~|~L(v)=v\}=L^V$. \\
{\em Well-defined at the second $V$} follows since from lemma \ref{SM1} we see that $M=N^{p^m-1}$ and thus $NM=0$.\\
{\em Well-defined at the second $V^L$}: $M(v)=v+Lv+\ldots+L^{v^m-1}v$ is invariant under $L$ so $\im(M)\subseteq V^L$. Exactness is trival.\\
For the last sentence: notice that $V^L$ is automatically a $k$-algebra. Now let us show that
$\im(M)$ is an ideal. Let $w\in \im(M)$, $v\in V^L$. Then there eixsts $u\in V: M(u)=w$. Now
$M(uv)=\sum_{i=0}^{p^m-1} L^i(uv)=$ (using that $L$ is a ring homomorphism and $L(v)=v$) $=\sum_{i=0}^{p^m-1} L^i(u)v=M(u)v$ hence $wv\in V^L$. Since $\im(M)$ is a linear subspace of $V^L$, we are done.
\end{proof}

In the next sections, $V=k^{[n]}$ and $L=F=(x_1+f_1,\ldots, x_n+f_n)\in \BA_n(k)$.

\subsection{Dimension 1}

Note that below, $N,M$ are as in lemma \ref{SM1} and definition \ref{SM1a}.
We have $n=1$ here,so  $F$ of order $p$ is equivalent to  $f_1\in k^*$.

\begin{lemma}  
\label{exact1}
If $F=(x_1+f_1)$ where $f_1\in k^*$, then $k[x_1]^F=k[\tilde{x}_1]$ where $\tilde{x}_1=x_1^p-f_1^{p-1}x_1$ and 
the sequences
\[ 0\lp k[\tilde{x}_1] \lp k[x_1] \overset{N}{\lp} k[x_1]  \overset{M}{\lp}  k[\tilde{x}_1] \lp 0\]
\[ 0\lp \im(N) \lp k[x_1]\overset{M}{\lp} k[x_1] \overset{N}{\lp} \im(N) \lp 0 \]
are exact.  \\
A representant system for $k[x_1]/\im(N)$ is $x_1^{p-1}k[\tilde{x}_1]$. Another representant system is $x_1^{p-1}k[x_1^p]$.
\end{lemma}

\begin{proof}
Since $\ord(F)=p$ this is a special case of \ref{SM3}.
We need to check a few things:\\
{\underline{(1) $k[x_1]^F=k[\tilde{x}_1]$.}} 
Note that $\tilde{x}_1\in k[x_1]^F$ indeed. Let $f\in k[x_1]^F$ be the lowest degree polynomial which is not in $k[\tilde{x}_1]$. 
Since we can reduce $f$ by $\tilde{x}_1^m$, we can assume that $\deg(f)=m$ is not a multiple of $p$.
So assume $f=a_mx_1^m+a_{m-1}x_1^{m-1}+\ldots$ where $a_i\in k$ and $(m,p)=1$. Then the coefficient of $x_1^{m-1}$ of $f(x_1+f_1)-f(x_1)$ turns out to be $a_m m f_1 $ which is nonzero. This is a contradiction, so there exists no such $f$, and thus  $k[x_1]^F=k[\tilde{x}_1]$.\\
\underline{(2) $\im(M)=k[\tilde{x}_1]=\ker(N)$.}
Note that $N$ decreases degree: hence,  $M(x_1^{p-1})=N^{p-1}(x_1^{p-1})\in k$. 
But the constant term is $\sum_{i=0}^{p} (x_1+i f_1)^{p-1}=\sum_{i=0}^p  i^{p-1} f_1^{p-1}=-f_1^{p-1}\in k^*$. 
Hence, $1\in \im(M)$ so the ideal $\im(M)=k[\tilde{x}_1]$. \\
\underline{(3) $\ker(M)/\im(N)=0$.} 
The sequence is now exact if $\ker(M)/\im(N)=0$. 
If the vector spaces would be finite dimensional, then this result follows from the fact that $\ker(N)=\im(M)$. 
We will restrict to
finite dimensional subspaces to conclude the result:
Define $k[x_1]_d=$ the set of polynomials of degree $d$ and less.
We claim that the sequence
\[ 0\lp  k[\tilde{x}_1]_d  \overset{i}{\hookrightarrow} k[x_1]_d \overset{N}{\lp} k[x_1]_d  \overset{M}{\lp} k[\tilde{x}_1]_d \lp 0 \]
is exact {\bf if $d=d_0p+p-1$} for some $d_0\in \N$  (it will NOT be exact for other $d$ !).
Now if $g\in k[\tilde{x}_1]$ then $M(x_1^{p-1}g)=g$; hence
$k[\tilde{x}_1]_{pd_0+p-1}
\supseteq M(k[x_1]_{d})\supseteq M(x_1^{p-1}\cdot k[\tilde{x}_1]_{pd_0})=k[\tilde{x}_1]_{pd_0}=k[\tilde{x}_1]_{pd_0+p-1}$ 
and thus  $\im(M|_{k[x_1]_d})=\ker(N|_{k[x_1]_d})$, a vector space of dimension $d_0$. 
Thus, $\dim(\ker(M|_{k[x_1]_d})=\dim(\im(N_{k[x_1]_d}))$ and we can conclude that since $\ker(M|_{k[x_1]_d})\supseteq 
\im(N_{k[x_1]_d})$ that they must be equal. If we now take unions, we get
\[ \ker(M)=\bigcup_{d\in p\Z+p-1} \ker(M|_{k[x_1]_d})=\bigcup_{d\in p\Z+p-1} \im(N_{k[x_1]_d})=\im(N) .\]
\underline{(4) The second sequence is exact} since $\ker(M)=\im(N)$ by the first exact sequence, and by the fact that $\im(M)=\ker(N)=k[\tilde{x}_1]$.

Now let us determine a \underline{(5) representant system} for $k[x_1]/\im(N)$. Since  $N(x_1^{p-1}g(\tilde{x}_1))=g(\tilde{x}_1)$, $\ker(N)\cap x_1^{p-1}k[\tilde{x}_1]=\{0\}$.
Thus $x_1^{p-1}k[\tilde{x}_1]$ is a representant system of $k[x_1]/\ker(N)=k[x_1]/\im(M)$.

Now notice that if $(m,p)=1$, then $\deg(N(x_1^m))=m-1$. This means that we have polynomials of all degrees $d$ as long as
$d\mod{p}\not =p-1$.
Note that $\deg(\tilde{x}_1)=p$, and notice that we thus have
$x_1^{p-1}\tilde{x}_1^m\mod(\ker(N)) = x_1^{p-1}(x_1^{pm}+a_{m-1}x_1^{p(m-1)}
+\ldots +a_1x_1^p+a_0)$ for some $a_i\in k$. These elements form a $k$-basis of a new representant system for $k[x_1]/\ker(N)$.
This means that $\{x_1^{p-1}(x_1^{pm})~|~m\in \N\}$ forms a $k$-basis of this new representant system,
i.e. $x_1^{p-1}k[x_1^p]$ is another representant system.
\end{proof}

\begin{corollary} \label{Ring}
Let $R$ be a commutative domain of characteristic $p$. 
If $F=(x_1+f_1)$ where $f_1\in R^*$, then $R[x_1]^F=R[\tilde{x}_1]$ where $\tilde{x}_1=x_1^p-f_1^{p-1}x_1$ and 
the sequences
\[ 0\lp R[\tilde{x}_1] \lp R[x_1] \overset{N}{\lp} R[x_1]  \overset{M}{\lp}  R[\tilde{x}_1] \lp 0\]
\[ 0\lp \im(N) \lp R[x_1]\overset{M}{\lp} R[x_1] \overset{N}{\lp} \im(N) \lp 0 \]
are exact.  \\
A representant system for $R[x_1]/\im(N)$ is $x_1^{p-1}R[\tilde{x}_1]$. Another representant system is $x_1^{p-1}R[x_1^p]$.
\end{corollary}

\begin{proof}
Obviously if $k$ is the quotient field of $R$, the result follows from lemma \ref{exact1}.
We need to check that intersecting from $k[x_1]$ to $R[x_1]$ everything goes well. 
First of all, $R[x_1]^F=k[x_1]^F\cap R[x_1]=k[\tilde{x_1}]\cap R[x_1]=R[\tilde{x}_1]$. 
Then, we need to check that the maps $N$ and $M$ do not miss things in their images; i.e. we need to check that $\im(N|_{R[x_1]})=\im(N|_{k[x_1]})\cap R[x_1]$, whereas a priori we only have $\subseteq$ in stead of $=$. (Note that  $\ker(N|_{R[x_1]})=\ker(N|_{k[x_1]})\cap R[x_1]$
as well as $\ker(M|_{R[x_1]})=\ker(M|_{k[x_1]})\cap R[x_1]$ trivially.)
$R[x_1]$ is a free $R$-module with basis $1,x_1,x_1^2,\ldots$. $M$ and $N$ send these basis elements into $\F_p(f_1)[x_1]$ (which is a subring of $R[x_1]$), and thus 
$\im(N|_{R[x_1]})=R\cdot\im(N|_{\F_p(f_1)[x_1]})=k\cdot\im(N|_{\F_p(f_1)[x_1]})\cap R[x_1]= \im(N|_{k[x_1]})\cap R[x_1]
= \ker(N|_{k[x_1]})\cap R[x_1]=\ker(N|_{R[x_1]})$. A similar proof for $M$. 

The results on the representant system follow by a similar argument:
A basis of the representant system over $k$ is $\{x_1^{p-1}\tilde{x}_1^i; i\in \N\}$. Then this can be used as a basis for 
the representant system of $R[x_1]/\im(N|_{R[x_1]})$ as well.
\end{proof}

\subsection{Induction step  $T_{\ref{T1}}[n]$ from $T_{\ref{T1}}[n-1]$ and $T_{\ref{T2}}[n-1]$}

In the rest of this section, we will consider $F=(x_1+f_1,\ldots, x_n+f_n)\in \BA_n(k)$ of order $p^n$. 
Define $F_i=x_i+f_i$ and $\vec{F}_i=(F_i,F_{i+1},\ldots, F_n)$.
We define $N=F-I$ and $M=\sum_{i=0}^{p^n-1} F^i$ as before, but also define 
\[ N_i:=\vec{F}_i-I_i = N|_{ k[x_i,\ldots, x_n]}\]
Notice that $\vec{F}_1=F, N_1=N$.

\begin{lemma} \label{lemmaP} Let $R$ be a domain of characteristic $p$. Let $F=(x_1+f_1,\ldots, x_n+f_n)\in \BA_n(R)$, where $f_n\in R^*$. 
Then $F^p=I$ if and only if $F$ can be conjugated by some $\tau\in \BA_{n-1}(R[x_n])$ to
$\tau^{-1}F\tau=(x_n+f_n)$. 
\end{lemma}

\begin{proof}
The ``if'' side is trivial. So let us assume $F^p=I$ and show we can conjugate $F$ to the given form.
Let us assume we can conjugate $F$ to $(x_1+f_1,\ldots, x_k+f_k, x_{k+1},\ldots, x_{n-1}, x_n+f_n)$.
We will consider this as a map $F'=(x_1+f_1,\ldots, x_k+f_k, x_n+f_n)$ on $(R[x_{k+1},\ldots, x_{n-1}])[x_1,\ldots,x_k,x_n]$. 
We get  $I=F^p=(\ldots, x_k+M(f_k), x_n+f_n)$, and thus apparently $M(f_k)=0$. Thus, 
$f_k\in \ker(M)$ which equals (using corollary \ref{Ring}) $\im(N)$. Let $g_k$ be such that $N(g_k)=f_k$. 
Then $(x_k-g_k)F (x_k+g_k)=(\ldots, x_k, x_{k+1},\ldots, x_{n-1}, x_n+f_n)$. Continuing this process, the lemma is proven. 
\end{proof}

\begin{lemma}
Assume  $T_{\ref{T1}}[n-1]$ and $T_{\ref{T2}}[n-1]$ . Then $T_{\ref{T1}}[n]$ holds.
\end{lemma}

\begin{proof}
We are thus considering  $F=(x_1+f_1,\ldots, x_n+f_n)\in \BA_n(k)$ where $f_n\in k^*$ and $\ord(F)=p^n$.
We want to prove $\inv(F)=k[\tilde{x}_1,\ldots, \tilde{x}_n]$ where 
$\tilde{x}_i=x_i^p+a_i^{p-1}x_i+b_i$ where $a_i,b_i\in k[x_{i+1},\ldots, x_n]$. 
The induction assumption $T_{\ref{T2}}[n-1]$ is used in the form of $\im(M_2)=\ker(N_2)$ and $\im(N_2)=\ker(M_2)$. 

Consider $F^{p^{n-1}}=(x_1+g_1,x_2,\ldots,x_n)$. Note that 
$0\not =g_1=M_2(f_1)\in \im(M_2)\subseteq \ker(N_2)\subseteq \ker(N)=\inv(F)$ (no appeal to assumptions here!) and thus $F(g_1)=g_1$. 
Now $\inv(F^{p^{n-1}})=k[x_1',x_2,\ldots,x_n]$ where $x_1'=x_1^p-g_1^{p-1}x_1$. 
Note that $\inv(F)\subseteq \inv(F^{p^{n-1}})$. We will restrict $F$ to $A:=\inv(F^{p^{n-1}})$
and compute $A^F=\inv(F)$. 
Now 
\[ F(x_1')=(x_1+f_1)^p-F(g_1)^{p-1}(x_1+f_1)=x_1'+f_1^p-g_1^{p-1}f_1.\]
Thus, $F|_A$ is triangular: $(x_1',x_2,\ldots, x_n)=(x_1'+g,x_2+f_2,\ldots, x_n+f_n)$
where $g=f_1^p-g_1^{p-1}f_1$. 

$F|_A$ has order $p^{n-1}$: it must be at least $p^{n-1}$ since $F|_A$ restricted to $k[x_2,\ldots,x_n]$ has order $p^{n-1}$.
If $h\in A$, then $F^{p^{n-1}}(h)=h$ by definition - so $F|_A$ is at most $p^{n-1}$. 

Now $F|_A=(x_1'+g, \tilde{F})$ where $\tilde{F}=F|_{k[x_2,\ldots, x_n]}$. 
$I|_A=F|_A^{p^{n-1}}=(x_1'+M_2(g), x_2,\ldots, x_n)$ and thus apparently $M_2(g)=0$, i.e. $g\in \ker(M_2)=\im(N_2)=k[x_2,\ldots,x_n]^{\tilde{F}}$ (here we have used $T_{\ref{T1}}[n-1]$). 
This means that we can find $h\in k[x_2,\ldots, x_n]$ such that $(x'_1+h)F|_A(x_1'-h)=(x_1', \tilde{F})$. 
Thus, $\inv(F|_A)=(x_1+h)\inv(x_1', \tilde{F})$.
Now $\inv(\tilde{F})=k[\tilde{x}_2,\ldots, \tilde{x}_n]$ as provided by induction. 
So $\inv(F|_A)=(x_1'+h)k[x_1',\tilde{x}_2,\ldots, \tilde{x}_n]=
k[x_1'+h, \tilde{x}_2,\ldots, \tilde{x}_n]$ and thus if we define $\tilde{x}_1=x_1'+h=x_1^p-g_1^{p-1}x_1+h$ we are done. 
\end{proof}

\begin{remark} In the proof above, we thus see that the $a_i$ in theorem \ref{T1} satisfy
$a_i=M_i(f_i)$. 
\end{remark}

\subsection{Induction step $T_{\ref{T2}}[n]$ from $T_{\ref{T2}}[n-1]$ and $T_{\ref{T1}}[n]$}

\begin{definition} \label{defVd} Assume $T_{\ref{T1}}[n]$. 
 Given $F=(x_1+f_1,\ldots, x_n+f_n)$, define (inductively) $\deg_F$ on $k^{[n]}$ by \\
(1) $\deg_F(x_n)=1$,\\
For $i=2$ to $n$ choose $\deg_F(x_i)$ large enough such that such that \\
(2a) $\deg_F(x_i)\geq \deg_F(f_i)$,\\
(2b) $\deg_F(x_i)\geq\deg_F( a_i), \deg_F(b_i)$ from theorem $T_{\ref{T1}}[n]$,  \\
(2c) $\deg_F(x_i)\in p\Z$.\\
Write $k^{[n]}_d=\{g\in k^{[n]} ~|~\deg_F(g)\leq d\}$. \\
Define $V_d:=k[x_n]_{p-1} k^{[n]}_{d} $. \\
Define ${W}_d:={V}_d\cap (k^{[n]})^F=V_d^F$.\\
\end{definition}

In the above definition, we have some choice in $\deg_F(x_i)$, but we can make it unique in stating that $\deg_F(x_i)$ should be as low as possible within the constraints (though we don't really care). 
The requirement (2a) is picked such that $F(k^{[n]}_d)\subseteq k^{[n]}_d$. 
The requirement (2b) is picked so  $\deg_F(\tilde{x}_i)=\deg_F(x_i^p)$, in order to be able to predict the degree of $\tilde{x}_i$. 
The requirement (2c) is added so that $\deg_F(x_n)$ is the only variable having degree coprime to $p$. 

Note that since $F(k^{[n]}_d)\subseteq k^{[n]}_d$,  we have a finite dimensional filtration of $k^{[n]}$ preserved by $F$ (see 
lemma \ref{SM1}).
Note that also $F(V_d)\subseteq V_d$, as $F(x_n^ig)=(x_n+f_n)^iF(g)\in V_d$ for every $g\in (k^{[n]})_d$ etc. so $V_d$ gives another such filtration of $k^{[n]}$.

\begin{lemma} \label{Wdeg} Assuming $T_{\ref{T1}}[n]$, we have $W_{pd}=(k^{[n]}_{pd})^F$.
\end{lemma}

\begin{proof}
If $(k^{[n]})^F=k[\tilde{x}_1,\ldots, \tilde{x}_n]$, then we see that  any element invariant under $F$ has degree a multiple of $p$, i.e. 
$W_{pa+b}\subseteq k^{[n]}_{pa}$ for each $a,b\in \N, 0\leq b\leq p-1$.
Since $V_{pd}\subseteq k^{[n]}_{pd+p-1}$ we get $W_{pd}=V_{pd}\cap (k^{[n]})^F\subseteq (k^{[n]}_{pd})$ and the result follows.
\end{proof}

\begin{lemma} 
Assume $T_{\ref{T1}}[n]$ and $T_{\ref{T2}}[n-1]$. Then $T_{\ref{T2}}[n]$ is true. So, 
let $F=(x_1+f_1,\ldots, x_n+f_n)$ such that $\ord(F)=p^n$. 
Then the sequence
\[ 0\lp (k^{[n]})^F \lp k^{[n]} \overset{N}{\lp} k^{[n]} \overset{M}{\lp}  (k^{[n]})^F \lp 0\]
is exact. 
Furthermore, $\im(M)=\ker(N)$. 
\end{lemma}

\begin{proof}
Notice first that $f_n\in k^*$ since otherwise $\ord(F)\leq p^{n-1}$. 

Using lemma \ref{SM3} we see that the only things to prove are  (1) $\ker(N)=(k^{[n]})^F$, (2)  $\im(M)=(k^{[n]})^F $,
(3) $\im(N)=\ker(M)$ and (4) $\im(M)=\ker(N)$. \\
(1) $\ker(N)=\{g\in k^{[n]} ~|~F(g)-g=0\}=(k^{[n]})^F$.\\
(2) $M(x_n^{p-1})=-f_n^{p-1}\in k^*$  (see part (2) of the proof of \ref{exact1} for a detailed computation), so the ideal $\im(M)=(k^{[n]})^F $. \\
(4) follows from (1) and (2).\\
(3) It is now tempting to state that since $(k^{[n]})^F =\ker(N)=\im(M)$, then $\im(N)=\ker(M)$, but since the $k$-vector spaces are infinite dimensional, this argument does not hold. However, what we will do, is restrict to finite dimensional subspaces $V_d$ and $W_d:=V_d\cap (k^{[n]})^F $ for which $\cup_d V_d=k^{[n]}$, {\em and} for which the restricted sequence
\[ 0\lp W_d\lp V_d  \overset{N}{\lp} V_d \overset{M}{\lp} W_d \lp 0 \]
is (a) well-defined, (b) exact. Note that there DO exist linear subspaces for which (a) holds but (b) not, so we need to define 
$V_d$ carefully - we claim that the definition in \ref{defVd} works for well-chosen $d$.
For this, note that $M(x_n^{p-1})\in k*$, and that if $f\in (k^{n]})^F$ and $g\in k^{[n]}$, then $M(fg)=fM(g)$. 
This means that $M(V_d)\supseteq M(x_n^{p-1} (k^{[n]}_d)^F)=(k^{[n]}_d)^F$.
Thus, using lemma \ref{Wdeg}
\[ W_{pd}=V_{pd}\cap (k^{[n]})^F\supseteq M(V_{pd}) \supseteq (k^{[n]}_{pd})^F=W_{pd}\] 
Thus, $\im(M|_{V_{pd}})=W_{pd}=\ker(N|_{V_{pd}})$ and the sequence
\[ 0\lp W_{pd}\lp V_{pd}\overset{N}{\lp} V_{pd} \overset{M}{\lp} W_{pd} \lp 0 \]
is exact. So (3) holds (and (1),(2), (3) yield that the sequence stated in the lemma is exact). 
\end{proof}

\subsection{Conjugacy classes: maximal order case in $\BA_n(k)$}

When we're conjugating $F$ by $(x_i+g_i)$, then the $i$-th component changes by $g_i-F(g_i)=N(g_i)$, which 
means we need to understand $\im(N)$ - or, better said, we want to have representants in $k^{[n]}$ of $k^{[n]}/\im(N)$.
For this, we need the following lemma, which builds on on theorems \ref{T1} and \ref{T2}.

\begin{lemma} \label{split}
Let $F=(x_1+f_1,\ldots, x_n+f_n)$ such that $\ord(F)=p^n$. Then
\[ k^{[n]}=x_n^{p-1}\inv(F)\oplus \im(N).\]
\end{lemma}

\begin{proof} Using theorem \ref{T2} we see that $\im(N)=\ker(M)$. We have the surjective map 
$k^{[n]}\overset{M}{\lp} \inv(F)$. We provide a section $s:\inv(F)\lp k^{[n]}$ by $s(f)=-f_n^{1-p}x_n^{p-1}f$. 
Indeed, since $M(x_n^{p-1})=-f_n^{p-1}$, we get that 
$M(-f_1^{1-p}x_n^{p-1}f)=f$ (see the argument at the end of the proof of \ref{SM3} for detailed reasoning). 
Thus, $Ms(f)=f$. This means that we can make a split exact sequence
\[ 0\lp \ker(M)\lp k^{[n]}\lp \im(M)\lp 0\]
and $k^{[n]}=\ker(M)\oplus s(\im(M))=\im(N)\oplus x_n^{p-1}\inv(F)$. 
\end{proof}

\begin{corollary}  (of lemma \ref{split}) \label{split2}
Let $F=(x_1+f_1,\ldots, x_n+f_n)$ such that $\ord(F)=p^n$. Then
\[ k^{[n]}=x_n^{p-1}k[x_1^p,\ldots,x_n^p]\oplus \im(N).\]
\end{corollary}

\begin{proof} 
Theorem \ref{T1} tells us that $\inv(F)=k[\tilde{x}_1,\ldots,\tilde{x}_n]$. 
Lemma \ref{split} tells us that a $k$-basis of a representant system of $k^{[n]}/\im(N)$ is $\{ x_n^{p-1} \tilde{x}_1^{a_1}\cdots \tilde{x}_n^{a_n} 
~|~ a_i\in \N\}$. 
Now let us use the standard lexicographic ordering $\dlex$ on $k^{[n]}$ by stating $x_1>>\ldots >> x_n$ (and beyond that the standard ordering, like  $\dlex(x_i^2)>\dlex(x_i)$). Then of any  $\tilde{x}_1^{a_1}\cdots \tilde{x}_n^{a_n}$, the leading term w.r.t. $\dlex$ is $x_1^{pa_1}\cdots x_n^{pa_n}$. 
This means that the following is also a $k$-basis of a representant system of $k^{[n]}/\im(N)$: 
$\{ x_n^{p-1} {x}_1^{pa_1}\cdots {x}_n^{pa_n} 
~|~ a_i\in \N\}$. 
In other words, we can pick $x_n^{p-1}k[x_1^p,\ldots,x_n^p]$ as a representant system, and the result follows. 
\end{proof}

\begin{corollary} \label{splitCC} (of theorem \ref{T1} and lemma \ref{split}.)
Let $F=(x_1+f_1,\ldots, x_n+f_n)$ such that $\ord(F)=p^n$.
Then $F$ is equivalent to exactly one 
$G=(x_1+f'_1,\ldots, x_{n-1}+f'_{n-1}, x_n+f_n)$ where $f_i'\in x_n^{p-1}(k^{[n]})^G$ if $1\leq i\leq n-1$. 
\end{corollary}

\begin{proof} 
By induction. The theorem is true for $n=1$. Assume the theorem is true for $n-1$.
This means we can assume $F=(x_1+f_1,x_2+f_2',\ldots, x_n+f_n')$, and that conjugation by $(x_i+g_i)$ where $2\leq i\leq n$ 
cannot be used anymore as it disturbes the form. Now, conjugation by $(x_1+g_1)$ makes it possible to change $f_1$ by elements of $\im(N)$. Using lemma \ref{split} we see that we can change to exactly one $f_1'\in x_n^{p-1}\inv(F)$. 
\end{proof}

The below is actually a corollary of corollary \ref{splitCC} and lemma \ref{split2}, but since the result is the most elegant one, we call it ``theorem'':

\begin{theorem} \label{splitCC2} 
Let $F=(x_1+f_1,\ldots, x_n+f_n)$ such that $\ord(F)=p^n$.
Then $F$ is equivalent to exactly one 
$G=(x_1+f'_1,\ldots, x_{n-1}+f'_{n-1}, x_n+f_n)$ where $f_i'\in x_n^{p-1}k[x_1^p,\ldots,x_n^p]$ if $1\leq i\leq n-1$. 
\end{theorem}

\begin{proof} The result follows directly from  corollary \ref{split2}, with the same proof as \ref{splitCC} required.\end{proof}

\noindent
{\bf Warning:} The above theorem does NOT state that any sequence $f'_1,\ldots, f'_{n-1},f_n$ 
where 
$f_i'\in x_n^{p-1}k[x_1^p,\ldots,x_n^p]$ if $1\leq i\leq n-1$, $f_n\in k^*$ 
gives a map of order $p^n$.

\begin{remark} In \cite{Mau12} a similar case was done for $k=\F_p$ a finite field with $p$ elements, and then 
not considering the triangular automorphisms, but the permutations $\F_p^n\lp \F_p^n$ induced by them.
It was shown that if $F=(x_1+f_1,\ldots,x_n+f_n)$ and additionally, $\deg_{x_i}(f_j)\leq p-1$ for each $i,j$ (which you may assume as you're only interested in the map), then $F$ is of order $p^n$ if and only if for all $1\leq i\leq n$, the coefficient of 
$(x_{i+1}\cdots x_n)^{p-1}$ is nonzero. 
This statement is {\em not} in contradiction with corollary \ref{splitCC2} - an example is $(x+y^pz^{p-1},y+z^{p-1},z+1)$ which 
has order $p^3$, but restricted to $\F_p^3$ has order $p^2$. 
\end{remark}

\begin{question} The representation of \ref{splitCC2} is pleasing to the eye, but perhaps not the best if your goal is to actually iterate such an element. In characteristic zero, we can write $F=\exp(D)$ and then $F^n=\exp(nD)$. It would be nice if there's a way to 
have such an elegant description of iterates of $F$ in this characteristic $p$ case too.
\end{question}

\section{Characteristic $p$: the generic 2-variable case}
\label{vier}
\label{Sec4}

In this whole section \ref{vier}, $k$ is still of characteristic $p$.

\subsection{Conjugacy classes of $\BA_2(k)$}

The previous section now makes it easy to determine the conjugacy classes of elements in $\BA_2(k)$. 
(We shun away from the generic ring case, as it would require a deeper understanding of the map $N$ over rings. 
In particular, we need to understand  how lemma \ref{split} behaves in that case.)
There are only three types of maps: order 1 (the identity), order $p$, and order $p^2$. 
Order $p^2$ is taken care of by theorem \ref{splitCC2}, while order $p$ is taken care of by lemma \ref{lemmaP}. Gathering up these results we get with very little extra effort the following corollary:

\begin{corollary} \label{CCx} The conjugacy classes in $\BA_2(k)$ are
\begin{itemize}
\item  $(x,y+\lambda)$ where $\lambda\in k$,
\item   $(x+f(y), y)$ where $f\not =0$, and a unique representant in $k[y]/k$ under the action $k\times k[y]\lp k[y]$ given by $\lambda\cdot g(y)\lp g(y+\lambda)$,
\item  $(x+y^{p-1}f(y^p), y+\lambda)$ where $\lambda\in k^*$ and $f\not =0$. 
\end{itemize}
\end{corollary}

\subsection{Conjugacy classes of $\BA_2(k)$ in $\BAA_2(k)$}

Now the result of corollary \ref{CCx} can now be slightly improved:

\begin{corollary} \label{CCxx} The conjugacy classes in $\BA_2(k)$ are
\begin{itemize}
\item $(x,y)$,
\item  $(x,y+1)$,
\item   $(x+f(y), y)$ where $f\not =0$ and monic, and a unique representant in $k[y]$ under the action $\GA_1(k) \times k[y]\lp k[y]$ given by $(\mu y+\lambda)\cdot g(y)\lp g(\mu y+\lambda)$,
\item  $(x+y^{p-1}f(y^p), y+1)$ where $\lambda\in k^*$ and $f\not =0$ monic.
\end{itemize}
\end{corollary}

\begin{proof}
The proof is very similar to the proof of \ref{ThCC2}, and in general easier since we're sticking with fields.
We just give a brief sketch of the two main points: \\
(1) ``monic'' in the fourth and third bullet point we get by conjugating with $(\lambda x)$ for suitable $\lambda$.\\
(2) The third bullet point is the only one where conjugation by $(\mu y+\lambda)$ does not change the second component, but can change the first component. Here, the ``characteristic $p$'' shows its head: we cannot use $y\lp y+\lambda$ to make sure that the $d-1$-th coefficent of $f$ is zero (where $d=\deg(f)$). So, we're essentially stuck in stating that $f$ should be picked unique under its equivalent forms under $y\lp \mu y+\lambda$. 
\end{proof}

\subsection{Conjugacy classes of $\BAA_2(k)$}

\begin{theorem} Let $k$ be a field of characteristic $p$. Then the conjugacy classes of $F=(F_1,F_2)$ in $\BAA_2(k)$ are
\begin{itemize}
\item Affine, 
\item   $(ax+f(y), y)$ where $a\in k^*$, $f\not =0$ and monic,  and a unique representant in $k[y]$ under the action $\GA_1(k) \times k[y]\lp k[y]$ given by $(\mu y+\lambda)\cdot g(y)\lp g(\mu y+\lambda)$,
\item  $(x+y^{p-1}f(y^p), y+1)$ where $\lambda\in k^*$ and $f\not =0$ monic.
\item $(a^dx+y^df(y^m), ay)$ where $m=\ord(a)$ ($m=0$ if $a$ no root of unity). $f\not = 0$ is monic. 
\end{itemize}
\end{theorem}

\begin{proof}
There are overlaps with the characteristic zero case, but we still give a complete proof.
If $F=(F_1,F_2)$ then we split up the cases we get from lemma \ref{G1}: $F_2=y, F_2=y+1$ and $F_2=ay$ where $a\not =1$. \\
(1)  $y$: Now $F_1=ax+f(y)$. Then we can make $f$ monic by conjugating with $(\mu x)$. Then, we can conjugate by $(\mu y+\lambda)$, and the result follows. 
which forces us  where $f(y)$ is monic. 
Again we can now conjugate by $y\lp \mu y+\lambda$ and we have no simpler way than just stating this.\\
(2) $y+1$: now $F=(ax+f(y), y+1)$. If $a=1$, then we're in the case of corollary \ref{CCxx}, and 
$F=(x+y^{p-1}f(y^p),y+1)$ where $f$ is monic or zero.
If $a\not =1$ then we need to consider what happens by conjugation with $(x+g(y))$. Then we can change $f$ by $ag(y)-g(y+1)$ and thus
we need to understand the map $g(y)\lp ag(y)-g(y+1)$. This map sends $y^d$ to $(a-1)y^d+$lower order terms. Since $a-1$ is invertible (as $a\not=1$), it is clear that this map is surjective (and actually bijective). Thus, we can conjugate to $(ax,y+1)$ which is affine. \\
(3) $ay$: If we conjugate $(bx+f(y),ay)$ by $(x+g(y))$ then we are modifying $f$ by $bg(y)-g(ay)$. 
We thus need to understand the map $g(y)\lp bg(y)-g(ay)$, which preserves monomials. The image of $y^d$ is $(b-a^d)y$, so 
we need to know when $b=a^d$; if this never happens, then we can conjugate to $(bx,ay)$ which is affine. So, write $b=a^d$ where $d\in N^*$ as minimal as possible. If $\ord(a)=m$ (automatically $d<m$) then we get $(bx+y^df(y^m), ay)$. If $a$ is no root of unity we get $(bx+\lambda y^d, ay)$. In both cases, we can conjugate by $(\mu x)$, and make sure that we have monic polynomials.
\end{proof}

\section{Finite order automorphisms in $\BA_2(k)$}

\label{Sec5}

In this section, we classify the finite order automorphisms for all fields $k$. We give a stand-alone proof (except for reference to \ref{exact1}) even though we could use the previous sections, and a slightly different classification, as we expect this to be of high interest.

\begin{lemma}
Let $F\in \GA_2(k)$ be of finite order $\ord(F)=s$, i.e. $F^s=I$.
Then $F$ can be conjugated to the following standard forms (unique up to as stated):
\begin{itemize}
\item[\textbf{A}] (affine)\\ An affine map (up to conjugation within the affine group)
\item[\textbf{U}] (unipotent) If characteristic $k$ is $p$, \\
 $(x+y^{p-1}f(y^p),y+1)$ where $f(y^p)\in k[y^p]$ monic nonzero.\\
$\ord(F)=p^2$.
\item[\textbf{M}] (mixed)\\ $(x+f(y^m), ay)$ where $\ord(a)=m$.
$\ord(F)=\lcm(p,m)$.
\item[\textbf{S}] (sequential)\\ $(a^mx+y^mf(y^{lm}),ay)$ where $a\in k^*$ satisfy $a^{ml}=1$ for some $m,l\in \N^*$ ($m$ chosen as small as possible, $l=\ord(a^m)>0$) and $f(y^{ml})\in k[y^{ml}]$ nonzero, and unique up to substitution  $y\lp \lambda y$ where $\lambda \in k^*$.
$\ord(F)=ml$.
\end{itemize}\end{lemma}

\begin{proof}
If $F \in \GA_2(k)$ of finite order, then using the Jung-v/d Kulk theorem it is easy to prove that $F$ must be up to conjugation either in $\Aff_2(k)$ or $\B_2(k)$.
We thus assume $F\in \B_2(k)\backslash \Aff_2(k)$. If one conjugates $F$ by some $G\in \GA_2(k)$ and $G^{-1}FG\in \B_2(k)$, then
$G\in \B_2(k)$ because of the same reason. We thus can consider the conjugacy classes within $\B_2(k)$.

Let us write $F=(ax+f(y),by+c)$. Since we are conjugating within $\B_2(k)$, we can first choose a unique form for $by+c$ within $\B_1(k)$. If $b\not =1$ then one can conjugate by $(y-\lambda) (by+c)(y+\lambda)=(by+c+(b-1)\lambda$ so choosing $\lambda=-c(b-1)^{-1}$ we get the standard form $by$. In case $b=1$, $c\not =0$ then we can conjugate  $(c^{-1} y)(y+c)(c y)=(y+1)$. Thus, we can assume $c=0$ or $b=1$.

\underline{\bf Case $c=0$}

$F=(ax+f(y),by)$.
Conjugate by $(x-g(y), y)(ax+f(y),by)(x+g(y),y)=(ax+f(y)+(ag(y)-g(by)),by)$. If $b^m\not =a$ for any $m\in \N$, then we can choose $g(y)$ such that $f(y)+(ag(y)-g(by))=0$ and achieve $(ax,by)$, an affine map. So:
{\bf We may assume $b=a^m$} for some $m\in \N$. In that case, the above conjugation can standardize $f(x)$ to a polynomial which is a linear combination of monomials $x^n$ such that $(bx)^n=ax^n$.

Assume $b\not =1$, i.e. $l:=\ord(b)>0$. Then  $n\in m+ml\Z$ where $l$ is such that $(a^m)^l=1$, i.e. we get $f(x)\in x^mk[x^{ml}]$. We can change this form a bit by conjugation with $(x,\lambda y)$, but we claim that beyond this, the form is unique:  if we conjugate by $(dx+g(y), ey+\lambda)$ then we see that $\lambda=0$ otherwise the form changed.
We can write
 $(dx+g(y), ey)=(x,ey)(dx,y)(x+d^{-1}g(y),y)$. We may ignore the conjugation by $(x,ey)$. The conjugation by $(dx,y)$ does not change the form. Then, the conjugation by $(x+d^{-1}g(y),y)$ either changes the form, or leaves it invariant (in case $g(y)\in y^mk[x^{lm}]$). Thus, this gives form {\bf S}.

Subcase $b=1$: we can get the form as in case {\bf M}. It is easy to check that this form cannot be improved by a conjugation within $\B_2(k)$.

\underline{\bf Case $b=1$, $c\not =0$:} We thus can assume $F=(ax+f(y), y+1)$.
We can conjugate by $(x+g(y),y)(ax+f(y), y+1)(x-g(y),y)=(ax+f(y)-ag(y)+g(y+1), y+1)$.
In case $a\not =1$, then the map $E:k[y]\lp k[y]$ given by $g(y)\lp -ag(y)+g(y+1)$ is surjective (as $\deg(E(x^m))=m$).
However, in case
 $a=1$, then we are considering the map $N:f(x)\lp f(x+1)-f(x)$ from lemma \ref{exact1}. We can thus change $f(x)$ by elements of $\im(N)$.  That same lemma shows that a representant system of $k[x]/\im(N)$ is $x^{p-1}k[x^p]$, so we may assume $f(x)$ is in here. We can conjugate by $(dx,y)$ to make sure that $f$ is monic.  We have obtained the form {\bf U}. We claim that this form $(x+y^{p-1}f(y^p),y+1)$ is unique. The argument is similar as before: if one conjugated by
$(dx+g(y), ey+\lambda)$ then $\lambda=0, e=1$ otherwise the form $y+1$ is destroyed, and then
$(dx+g(y),y)$. Conjugating with this either destroys the standard form or leaves it invariant.
\end{proof}

\section{Further research}
\label{Sec6}

We gather up a list of future research questions.

\begin{itemize}
\item For rings $R$ containing $\Q$, determine the conjugacy classes of $\BA_3(R)$ and $\BAA_3(R)$.
Determine the conjugacy classes of $\BA_4(k)$ and $\BAA_4(k)$ where $k$ is a field of characteristic zero.
\item For rings $R$ not containing $\Q$ (i.e. having prime ring $\Z$ or $\Z_n$), determine the conjugacy classes of 
$\BA_2(R)$ and $\BAA_2(R)$. 
\item Find representants of the conjugacy classes which are relatively easy to iterate. In characteristic zero, we can do this by noticing that
$\exp(mD)=F^m$, but in characteristic $p$ we have an open question here for the elements of $\BA_2(k)$, and for 
maximal order maps in $\BA_n(k)$. See also \cite{Mau12}.
\item If $\kar(k)=p$, $F\in \GA_n(k)$ satisfies $F^{p^n}=I$ (i.e. $F$ is unipotent), are there similar theorems as \ref{T1} and \ref{T2} true for this general case? 
\end{itemize}

{\bf Acknowledgement} The author would like to thank  Immanuel Stampfli and J\'er\'emy Blanc for some interesting discussions which originated this paper, 
and in particular Blanc for the first proof of an early version of corollary \ref{exact1}.

\end{document}